\documentclass[a4paper, 12pt, oneside, reqno]{amsart}
\calclayout

\usepackage[utf8]{inputenc}
\usepackage[T1]{fontenc}
\usepackage[english]{babel}
\usepackage{amsmath, amsfonts, amssymb, amsthm, mathtools, indentfirst, bbm}
\usepackage[dvipsnames]{xcolor}
\usepackage[shortlabels]{enumitem}
\usepackage[languagenames,fixlanguage]{babelbib}
\usepackage{commath}
\usepackage{textcomp}
\usepackage{mathrsfs}
\usepackage{todonotes}
\usepackage{pgfplots}
\usepackage{subcaption}
\usepackage{graphicx}
\usepackage{cite}
\usepackage{placeins}
\allowdisplaybreaks

\usepackage[unicode=true, pdfstartview={XYZ null null 1.00}, colorlinks=true, linkcolor=blue, urlcolor=blue, citecolor=purple]{hyperref}

\makeatletter
\@namedef{subjclassname@2020}{\textup{2020} Mathematics Subject Classification}
\makeatother

\usepackage{array}

\newtheorem{theorem}{Theorem}[section]
\newtheorem{proposition}[theorem]{Proposition}
\newtheorem{lemma}[theorem]{Lemma}
\newtheorem{corollary}[theorem]{Corollary}
\theoremstyle{remark}
\newtheorem{remark}[theorem]{Remark}
\newtheorem{definition}[theorem]{Definition}

\newcommand{\E}{\mathbb{E}}
\newcommand{\Prob}{\mathbb{P}}
\newcommand{\Var}{\mathrm{Var}}
\newcommand{\1}{\mathbf{1}}

\newcommand{\N}{\mathbb{N}}

\numberwithin{equation}{section}

\makeatletter
\def\@tocline#1#2#3#4#5#6#7{\relax
  \ifnum #1>\c@tocdepth 
  \else
    \par \addpenalty\@secpenalty\addvspace{#2}%
    \begingroup \hyphenpenalty\@M
    \@ifempty{#4}{%
      \@tempdima\csname r@tocindent\number#1\endcsname\relax
    }{%
      \@tempdima#4\relax
    }%
    \parindent\z@ \leftskip#3\relax \advance\leftskip\@tempdima\relax
    \rightskip\@pnumwidth plus4em \parfillskip-\@pnumwidth
    #5\leavevmode\hskip-\@tempdima
      \ifcase #1
       \or\or \hskip 1em \or \hskip 2em \else \hskip 3em \fi%
      #6\nobreak\relax
    \dotfill\hbox to\@pnumwidth{\@tocpagenum{#7}}\par
    \nobreak
    \endgroup
  \fi}
\makeatother

\title{Weighted-threshold Coupon Collection}
    \author{Sebastian M\"uller}
		\address{
		Sebastian M\"uller,
		Aix Marseille Universit\'e, CNRS, Centrale Marseille, I2M - UMR 7373, 13453 Marseille, France 
		}
		\email{sebastian.muller@univ-amu.fr}
	\author{Stjepan Šebek}
		\address{
		Stjepan Šebek,
		University of Zagreb Faculty of Electrical Engineering and Computing, 10000 Zagreb, Croatia}
		\email{stjepan.sebek@fer.unizg.hr}
\date{}

\begin{document}

\begin{abstract}
We study a weighted-threshold version of the coupon collector problem in
continuous time. Each type \(i\) is discovered at rate \(\lambda p_i\) and,
once discovered, contributes weight \(w_i\), where \(p\) and \(w\) are
probability vectors. The stopping time  when the total weight of the discovered
types first exceeds a fixed threshold \(\theta\in(0,1)\) is called the quorum time. 

We first prove concentration estimates and compare the quorum time with the
corresponding deterministic threshold time obtained from the mean discovered
weight. When all discovery rates are equal and the largest individual weight
tends to zero, the first-order asymptotics are universal and do not depend on
the weight vector. We then analyze the aligned Zipf family
\(p_i=w_i\propto i^{-s}\). This model has three regimes: a deterministic
linear scale for \(0\le s<1\), a critical scale \(H_NN^\theta\) at \(s=1\),
with an explicit leading constant, and a non-degenerate random hitting-time
limit for \(s>1\). Finally, we show that the expected quorum time need not be
monotone in the Zipf exponent. 
\end{abstract}


\maketitle

\section{Introduction}

The classical coupon collector problem asks how long it takes to observe every type in a finite population. In many stochastic sampling problems, however, complete coverage is not the relevant stopping criterion. One may only need to observe enough total mass: enough probability mass in a sampling problem, enough sensor coverage in a monitoring system, or enough voting weight in a quorum system. This leads to a natural weighted-threshold version of coupon collection. Each type has a weight, types are discovered at possibly unequal rates, and the process stops when the total weight of distinct discovered types crosses a prescribed threshold.

A central feature of the model is that two different profiles are present. The activity profile \(p\) controls how quickly types are discovered, while the weight profile \(w\) controls how much discovered types contribute to the threshold. These do not need to coincide, i.e.\ a frequently observed type may carry little weight, and a large-weight type may be rarely observed, and vice-versa.

We study this problem in continuous time. For \(1\le i\le N\), type \(i\) carries quorum weight \(w_i\) and is discovered at the first event time \(T_i\) of an independent Poisson clock of rate \(\lambda p_i\). Here \(w=(w_i)_{i=1}^N\) and \(p=(p_i)_{i=1}^N\) are probability vectors. The \emph{discovered-mass process} is
\begin{equation}\label{eq:def_of_discovered-mass_process}
    M_N(t) := \sum_{i=1}^N w_i \1_{\{T_i\le t\}}, \qquad t \ge 0,
\end{equation}
and the \emph{weighted quorum time} is
\begin{equation}\label{eq:def_of_quorum_time}
    \tau_{N,\theta}:=\inf\{t\ge 0:M_N(t)>\theta\},
    \qquad 0<\theta<1.
\end{equation}
Notice that the word ``distinct'' is important since repeated observations of the same type do not add new mass. Thus, the problem is close to the unequal-probability coupon collection, but the stopping rule is different. We do not wait for all types to appear, nor for a fixed number of types, nor for prescribed quotas of each type. We stop when the discovered set has enough total weight. The main focus of our paper is asymptotic analysis of the weighted quorum time $\tau_{N,\theta}$ as the number of different types $N$ grows to infinity.

The paper is organized as follows. First, in Section \ref{sec:Model_and_Exact_Representation}, we introduce the model rigorously. 
In Section~\ref{sec:Concentration_and_Deterministic_Comparison} we prove
concentration bounds for \(M_N(t)\) and deterministic comparison estimates
between \(\tau_{N,\theta}\) and the root of the mean profile.

The discovered-mass process \(M_N(t)\) from \eqref{eq:def_of_discovered-mass_process}, at each fixed time, is a weighted sum of independent Bernoulli variables, and Bernstein concentration gives a  comparison between the quorum time \(\tau_{N,\theta}\) introduced in~\eqref{eq:def_of_quorum_time} and the deterministic root of the mean profile
\[
    m_N(t):=\E[M_N(t)].
\]
In particular, in the so-called \emph{diffuse} setting, where the maximum weight tends to zero as the number of types $N$ grows to infinity (see Definition~\ref{def:diffuse_vs_atomic}), convergence of the mean profile implies a law of large numbers for the quorum time.

We then identify two benchmark regimes. Under homogeneous clocks, considered in Section~\ref{sec:Homogeneous _Clocks}, where \(p_i=1/N\), the first-order quorum time is universal for all diffuse weight profiles: 
\[
    \frac{\tau_{N,\theta}}{N}
    \xrightarrow[N\to\infty]{\mathbb P}
    -\frac{\log(1-\theta)}{\lambda}.
\]
The assumption of a diffuse weight profile is essential. If a type keeps positive weight as \(N\to\infty\), its random discovery time may remain visible on the scale \(N\), and the limit need not be deterministic.

This motivates the aligned Zipf family
\[
    p_i=w_i=\frac{i^{-s}}{\sum_{j=1}^N j^{-s}}.
\]
Here the same parameter controls discovery rates and weights. For \(s\le 1\), the largest weight tends to zero, while for \(s>1\) the leading weights have positive limits. The model therefore passes from deterministic first-order behaviour to a genuinely random limiting hitting time.

We prove that this transition has three parts. 
For \(0\le s<1\) (see Section~\ref{sec:The_Diffuse_Aligned_Zipf_Regime}), the quorum time $\tau_{N, \theta}$ converges, on the linear scale $N$, to a constant that is the solution of an explicit integral equation. At \(s=1\) (see Section~\ref{sec:The_Critical_Aligned_Zipf_Regime}), the correct scale is \(H_NN^\theta\), where \(H_N=\sum_{i=1}^N i^{-1}\), and we identify the leading constant. For \(s>1\) (see Section~\ref{sec:The_Atomic_Aligned_Zipf_Regime}), the quorum time converges almost surely to a random hitting time of an infinite weighted Bernoulli process. Finally, we show that the expected quorum time is not, in general, monotone in the Zipf exponent (see Section~\ref{sec:Exact_Non-Monotonicity}).

Weighted quorum formation is one practical source of the problem. Classical quorum models often count distinct participants equally \cite{gifford1979,malkhi1998,castro1999}; weighted variants arise when participants carry voting power, stake, reputation, resource responsibility, or another score \cite{kiayias2017,pass2017,muller2022tangle,muller2021fast}. In this interpretation a coupon type $i$ corresponds to a participant or node $i$, \(p_i\) describes activity or responsiveness of node $i$, \(w_i\) describes quorum weight of node $i$, and the quorum time is a tractable stochastic proxy for the time needed to accumulate enough distinct participating weight.

\subsection{Related work}
Unequal-probability coupon collection is a classical problem, with analytic, probabilistic, and computational treatments of the full-collection problem in \cite{flajolet1992,boneh1997,brown2008,neal2008,doumas2012,doumas2019}. Other nearby directions include the time to collect a fixed number of distinct coupons \cite{anceaume2015,anceaume2016}, a version of the full-collection problem where in each step a subset of all the coupons is sampled (instead of just one coupon) \cite{adler2001,baake2026}, multiple-subset stopping rules where there are several subsets of coupon types and coupons are collected until all of the types of at least one of these subsets have been collected \cite{chang2007}, unequal-probability full-collection approximations \cite{berenbrink2009}, and collecting until each coupon reaches a prescribed quota \cite{newman1960double, may2008coupon}.

Although coupon collection is a classical topic, recent work continues to study
new variants and asymptotic regimes, including group-draw models,
non-monotone dynamics, competition between collectors, and sibling variants
\cite{long2026, berend_sher2026, long2026-2, attrill-garoni,
doumas-spektor-1, barak-berend, long2026-3, long2026-4,
doumas-spektor-2}.

Our model fits the framework of multiple-subset stopping rule from \cite{chang2007} at the level of abstract definition. Let
\[
\mathcal{F}_\theta:=\left\{A\subseteq[N]:\sum_{i\in A}w_i>\theta\right\},
\]
where we use standard notation $[N] = \{1, 2, \ldots, N\}$. In terms of the multiple-subset stopping rule, we can say that we are collecting coupons until all of the types of at least one of the subsets from family $\mathcal{F}_\theta$ have been collected. However, we go beyond results from \cite{chang2007} in the direction of the weighted-threshold geometry. Notice that in our case the family $\mathcal{F}_\theta$ is generated by a weight vector, changes with $N$, and interacts with a possibly different activity profile $p$.

The aligned case \(p=w\) is also closely related to occupancy and missing-mass problems. Under Poissonized sampling, \(1-M_N(t)\) is the mass of types not yet observed. Infinite urn schemes and missing-mass concentration provide a natural probabilistic background, especially under power-law assumptions \cite{karlin1967,gnedin2007,berend2013,benhamou2017}. Our results are complementary to this literature. Namely, instead of estimating the missing mass at a fixed sample size, we study the first time at which the discovered mass crosses a given threshold.

\section{Model and Exact Representation}\label{sec:Model_and_Exact_Representation}
Fix $N \in \N$, a quorum-weight vector $w^{(N)}=(w_1,\dots,w_N)$ and an activity vector $p^{(N)}=(p_1,\dots,p_N)$ with
\[
w_i>0,
\qquad
\sum_{i=1}^N w_i=1,
\]
\[
p_i>0,
\qquad
\sum_{i=1}^N p_i=1,
\]
a rate parameter $\lambda>0$, and a threshold $\theta\in(0,1)$. For each $i\in[N]$, let $\{N_i(t)\}_{t\ge 0}$ be an independent Poisson process of rate $\lambda p_i$, and let
\[
T_i:=\inf\{t\ge 0:N_i(t)\ge 1\}.
\]
Equivalently,
\[
T_i \sim \mathrm{Exp}(\lambda p_i),
\qquad i\in[N],
\]
independently.
\begin{remark}\label{rem:bernoulli_structure}
    Notice that the discovered-mass process at time $t$, $M_N(t) = \sum_{i=1}^N w_i \1_{\{T_i \le t\}}$, is just a weighted sum of independent Bernoulli random variables. We will sometimes write $\xi_i(t) = \1_{\{T_i \le  t\}}$. Random variables $\{\xi_i(t) : 1 \le i \le N\}$ are clearly independent and have Bernoulli's distribution with
    \begin{equation*}
        \Prob(\xi_i = 1) = \Prob(T_i \le t) = 1 - e^{-\lambda p_i t}.
    \end{equation*}
\end{remark}
Since $M_N(t)$ is non-decreasing and right continuous, we have, for every $t\ge 0$,
\[
\{\tau_{N,\theta}>t\}=\{M_N(t)\le \theta\}.
\]
Applying Tonelli's theorem to the nonnegative random variable $\tau_{N,\theta}$, we get the standard integral formula for the expected value, namely,
\begin{equation}\label{eq:E_tau_Tonelli}
    \E[\tau_{N,\theta}] =
    \int_0^\infty \Prob(\tau_{N,\theta}>t)\,dt
    =
    \int_0^\infty \Prob(M_N(t)\le \theta)\,dt.    
\end{equation}
We now derive a formula for  $\Prob(M_N(t) \le \theta)$. For any $A \subseteq [N]$ denote by
\[
w(A):=\sum_{i\in A} w_i.
\]
Using Remark \ref{rem:bernoulli_structure} we get (for every $t\ge 0$)
\begin{align*}
    \Prob\bigl(M_N(t)\le \theta\bigr)
    & = \sum_{A\subseteq[N]:\,w(A)\le \theta}
\Prob\!\left(\{i:T_i\le t\}=A\right) \\
    & = \sum_{A\subseteq[N]:\,w(A)\le \theta}
\prod_{i\in A}\bigl(1-e^{-\lambda p_i t}\bigr)
\prod_{j\notin A} e^{-\lambda p_j t}.
\end{align*}
Using again Remark \ref{rem:bernoulli_structure} we obtain
\begin{equation}\label{eq:mean_profile}
   m_N(t) =\E[M_N(t)] = \sum_{i = 1}^N w_i \bigl(1-e^{-\lambda p_i t}\bigr),
\end{equation}
and
\begin{equation}\label{eq:Var_MNt}
    \Var(M_N(t))=\sum_{i=1}^N w_i^2 e^{-\lambda p_i t}\bigl(1-e^{-\lambda p_i t}\bigr).
\end{equation}

\section{Concentration and Deterministic Comparison}\label{sec:Concentration_and_Deterministic_Comparison}
In this section we provide a general concentration framework for weighted quorum collection with separate activity and quorum profiles. 
Let \(w^{(N)}=(w_1^{(N)},\ldots,w_N^{(N)})\) be a sequence of probability
vectors, and extend it by setting \(w_i^{(N)}=0\) for \(i>N\). Write
\[
    w_N^\ast:=\max_{1\le i\le N} w_i^{(N)} .
\]

\begin{definition}\label{def:diffuse_vs_atomic}

We say that the weights are \emph{diffuse} if
\[
    w_N^\ast \xrightarrow[N\to\infty]{} 0 .
\]

We say that the weights are \emph{atomic} if there exists a sequence \(w^{(\infty)}=(w_i^{(\infty)})_{i\ge1}\) such that
\[
    w_i^{(N)} \xrightarrow[N\to\infty]{} w_i^{(\infty)} \qquad\text{for every fixed } i\ge1,
\]
and
\[
     w_i^{(\infty)} >0 \mbox{ for some } i.
\]
\end{definition}

\begin{remark}
The two notions, diffuse and atomic, are not meant to cover all possible asymptotic behaviors. The diffuse case means that no individual type keeps positive weight. The atomic case means that at least one fixed type keeps positive limiting weight. In the aligned Zipf family studied below, \(0\le s\le1\) is diffuse, whereas \(s>1\) is atomic with $w_i^{(\infty)}=\frac{i^{-s}}{\zeta(s)}$.
\end{remark}

\begin{theorem}\label{thm:bernstein}
For every $t\ge 0$ and every $x>0$,
\[
\Prob\bigl(|M_N(t)-m_N(t)|\ge x\bigr)
\le
2\exp\!\left(
-\frac{x^2}{2\Var(M_N(t))+\frac23 w_N^\ast x}
\right).
\]
In particular, since $\Var(M_N(t))\le w_N^\ast$, one has
\[
\Prob\bigl(|M_N(t)-m_N(t)|\ge x\bigr)
\le
2\exp\!\left(
-\frac{x^2}{2w_N^\ast+\frac23 w_N^\ast x}
\right).
\]
\end{theorem}

\begin{proof}
Using the notation from Remark \ref{rem:bernoulli_structure} set
\[
X_i(t):=w_i\bigl(\xi_i(t)-\E[\xi_i(t)]\bigr),
\qquad i\in[N].
\]
Then $M_N(t)-m_N(t)=\sum_{i=1}^N X_i(t)$, the variables $X_i(t)$ are independent and centered, and
\[
|X_i(t)|\le w_i\le w_N^\ast.
\]
The sum of their variances is exactly $\Var(M_N(t))$. Bernstein's inequality (see \cite[Theorem 2.10]{BLM-concentration}), therefore, yields
\[
\Prob\bigl(|M_N(t)-m_N(t)|\ge x\bigr)
\le
2\exp\!\left(
-\frac{x^2}{2\Var(M_N(t))+\frac23 w_N^\ast x}
\right).
\]
For the simplified bound, note that from \eqref{eq:Var_MNt} we have
\[
\Var(M_N(t))
=
\sum_{i=1}^N w_i^2 e^{-\lambda p_i t}\bigl(1-e^{-\lambda p_i t}\bigr)
\le
\sum_{i=1}^N w_i^2
\le
w_N^\ast\sum_{i=1}^N w_i
=
w_N^\ast.
\qedhere
\]
\end{proof}

From \eqref{eq:mean_profile} we have that the \emph{mean profile} $m_N(t)$ is continuous and  strictly increasing in $t$, starts at $0$, and tends to $1$ as $t\to\infty$. Hence for every $\theta\in(0,1)$ there is a unique deterministic comparison time
\[
t_{N,\theta}:=m_N^{-1}(\theta).
\]

\begin{theorem}\label{thm:root-comparison}
Fix $\theta\in(0,1)$ and $0<\delta<t_{N,\theta}$. Define
\[
\Delta_{N,\theta}(\delta):=
\min\Bigl\{
\theta-m_N(t_{N,\theta}-\delta),
\,
m_N(t_{N,\theta}+\delta)-\theta
\Bigr\}.
\]
Then $\Delta_{N,\theta}(\delta)>0$ and
\[
\Prob\bigl(|\tau_{N,\theta}-t_{N,\theta}| > \delta\bigr)
\le
4\exp\!\left(
-\frac{\Delta_{N,\theta}(\delta)^2}
{2w_N^\ast+\frac23 w_N^\ast \Delta_{N,\theta}(\delta)}
\right).
\]
\end{theorem}

\begin{proof}
Since $m_N$ is strictly increasing and $m_N(t_{N,\theta})=\theta$, one has
\[
m_N(t_{N,\theta}-\delta)<\theta<m_N(t_{N,\theta}+\delta),
\]
so $\Delta_{N,\theta}(\delta)>0$. If $\tau_{N,\theta} < t_{N,\theta}-\delta$, then by monotonicity of $M_N$,
\[
M_N(t_{N,\theta}-\delta)>\theta.
\]
Hence
\[
M_N(t_{N,\theta}-\delta)-m_N(t_{N,\theta}-\delta)\ge
\theta-m_N(t_{N,\theta}-\delta)\ge \Delta_{N,\theta}(\delta).
\]
Similarly, if $\tau_{N,\theta} > t_{N,\theta}+\delta$, then
\[
M_N(t_{N,\theta}+\delta)\le \theta,
\]
so
\[
m_N(t_{N,\theta}+\delta)-M_N(t_{N,\theta}+\delta)\ge
m_N(t_{N,\theta}+\delta)-\theta\ge \Delta_{N,\theta}(\delta).
\]
Therefore
\begin{align*}
\Prob\bigl(|\tau_{N,\theta}-t_{N,\theta}| > \delta\bigr)
&\le
\Prob\Bigl(M_N(t_{N,\theta}-\delta)-m_N(t_{N,\theta}-\delta)\ge \Delta_{N,\theta}(\delta)\Bigr)
\\
&\quad+
\Prob\Bigl(m_N(t_{N,\theta}+\delta)-M_N(t_{N,\theta}+\delta)\ge \Delta_{N,\theta}(\delta)\Bigr).
\end{align*}
Applying Theorem~\ref{thm:bernstein} to each term gives
\[
\Prob\bigl(|\tau_{N,\theta}-t_{N,\theta}| > \delta\bigr)
\le
4\exp\!\left(
-\frac{\Delta_{N,\theta}(\delta)^2}
{2w_N^\ast+\frac23 w_N^\ast \Delta_{N,\theta}(\delta)}
\right).
\qedhere
\]
\end{proof}

\begin{corollary}\label{cor:root-comparison}
Let $(\delta_N)_{N \in \N}$ be a positive sequence such that $\delta_N<t_{N,\theta}$ for all sufficiently large $N$ and
\[
\frac{\Delta_{N,\theta}(\delta_N)^2}{w_N^\ast}\xrightarrow[N\to\infty]{}\infty.
\]
Then
\[
\lim_{N \to \infty} \Prob(|\tau_{N,\theta}-t_{N,\theta}| > \delta_N) = 0.
\]
\end{corollary}

\begin{proof}
Since $0<\Delta_{N,\theta}(\delta_N)\le 1$, the exponent in
Theorem~\ref{thm:root-comparison} satisfies
\[
\frac{\Delta_{N,\theta}(\delta_N)^2}
{2w_N^\ast+\frac23 w_N^\ast\Delta_{N,\theta}(\delta_N)}
=
\frac{\Delta_{N,\theta}(\delta_N)^2/w_N^\ast}
{2+\frac23\Delta_{N,\theta}(\delta_N)}
\xrightarrow[N\to\infty]{}\infty.
\]
The conclusion follows from Theorem~\ref{thm:root-comparison}.
\end{proof}

We are now ready to prove a general diffuse law of large numbers.
\begin{theorem}\label{thm:general-diffuse}
Assume that quorum weights are diffuse. Fix $\theta\in(0,1)$ and suppose there exist positive scales $(a_N)_{N\in \N}$ and a continuous, strictly increasing function $F:[0,\infty)\to[0,1)$ with
\[
F(0)=0,
\qquad
\lim_{\alpha\to\infty}F(\alpha)=1,
\]
such that for every $\alpha\ge 0$,
\[
m_N(a_N\alpha)\xrightarrow[N\to\infty]{} F(\alpha).
\]
Let $\alpha_\theta$ be the unique solution of $F(\alpha_\theta)=\theta$. Then
\[
\frac{\tau_{N,\theta}}{a_N}\xrightarrow[N\to\infty]{\Prob}\alpha_\theta.
\]
\end{theorem}

\begin{proof}
Fix $\varepsilon>0$ and set
\[
\alpha_-:=\max\{\alpha_\theta-\varepsilon,0\},
\qquad
\alpha_+:=\alpha_\theta+\varepsilon.
\]
Since $F$ is strictly increasing and continuous,
\[
F(\alpha_-)<\theta<F(\alpha_+).
\]
Choose
\[
\eta:=\frac12\min\{\theta-F(\alpha_-),\,F(\alpha_+)-\theta\}>0.
\]
By the assumed convergence of the mean profile, for all sufficiently large $N$,
\[
m_N(a_N\alpha_-)\le \theta-\eta,
\qquad
m_N(a_N\alpha_+)\ge \theta+\eta.
\]
Hence
\begin{align*}
\Prob\!\left(\tau_{N,\theta}\le a_N\alpha_-\right)
&=
\Prob\!\left(M_N(a_N\alpha_-)>\theta\right) \\
&\le
\Prob\!\left(M_N(a_N\alpha_-)-m_N(a_N\alpha_-)\ge \eta\right)
\xrightarrow[N\to\infty]{} 0,
\end{align*}
and similarly
\[
\Prob\!\left(\tau_{N,\theta}> a_N\alpha_+\right)\xrightarrow[N\to\infty]{} 0,
\]
because Theorem~\ref{thm:bernstein} and $w_N^\ast\xrightarrow[N\to\infty]{} 0$ imply
\[
\Prob\bigl(|M_N(t)-m_N(t)|\ge \eta\bigr)\xrightarrow[N\to\infty]{} 0
\]
uniformly in $t$. Therefore
\[
\Prob\!\left(\alpha_\theta-\varepsilon\le \frac{\tau_{N,\theta}}{a_N}\le \alpha_\theta+\varepsilon\right)\xrightarrow[N\to\infty]{} 1.
\]
Since $\varepsilon>0$ was arbitrary, the claim follows.
\end{proof}

\section{Homogeneous Clocks}\label{sec:Homogeneous _Clocks}
Assume throughout this section that
\[
p_i=\frac1N,
\qquad i\in[N].
\]
The quorum weights remain arbitrary, but diffuse (see Definition \ref{def:diffuse_vs_atomic}).

\begin{theorem}\label{thm:homogeneous}
Fix $\theta\in(0,1)$ and assume the quorum weights are diffuse. Then
\[
\frac{\tau_{N,\theta}}{N}\xrightarrow[N\to\infty]{\Prob}\frac{-\log(1-\theta)}{\lambda}.
\]
\end{theorem}

\begin{proof}
For every $t\ge 0$, by \eqref{eq:mean_profile}, we have
\[
m_N(t)
=
\sum_{i=1}^N w_i\bigl(1-e^{-\lambda t/N}\bigr)
=
1-e^{-\lambda t/N}.
\]
Hence, with $a_N=N/\lambda$,
\[
m_N(a_N\alpha)=1-e^{-\alpha}
\qquad \text{for all }\alpha\ge 0.
\]
The function
\[
F(\alpha):=1-e^{-\alpha}
\]
is continuous, strictly increasing, satisfies $F(0)=0$, and tends to $1$ as $\alpha\to\infty$. Its unique root at level $\theta$ is
\[
\alpha_\theta=-\log(1-\theta).
\]
Theorem~\ref{thm:general-diffuse} therefore applies and yields the claim.
\end{proof}

\begin{remark}
For homogeneous clocks, \(m_N(t)=1-e^{-\lambda t/N}\), independently of \(w\).
The condition \(\max_i w_i\to0\) is what turns this mean comparison into a
law of large numbers. Within this class of weights, the leading term of
\(\tau_{N,\theta}\) depends only on \(\theta\) and \(\lambda\).
\end{remark}

\section{The Diffuse Zipf Regime: $0\le s<1$}\label{sec:The_Diffuse_Aligned_Zipf_Regime}
Fix $s\in[0,1)$ and define
\[
q_{i,N}^{(s)}:=\frac{i^{-s}}{H_N^{(s)}},
\qquad
H_N^{(s)}:=\sum_{j=1}^N j^{-s}.
\]
We consider the aligned Zipf model where
\[
p_i=w_i=q_{i,N}^{(s)},
\qquad i\in[N].
\]
Let $\tau_{N,\theta}^{(s)}$ and $M_N^{(s)}$ be the corresponding quorum time and discovered-mass process.
\begin{remark}
    Notice that
    \begin{equation*}
        w_N^\ast = \max_{1 \le i \le N} w_i = \frac{1}{H_N^{(s)}} \xrightarrow[N \to \infty]{} 0,
    \end{equation*}
    since $H_N^{(s)}=\sum_{j=1}^N j^{-s}\to\infty$ for $s<1$.
\end{remark}

As is shown in Theorem \ref{thm:diffuse} below, the behavior of the quorum time in this regime is still linear (as in the homogeneous clock case). The heuristic behind it is that for $0\le s<1$, the quorum weight remains spread across a macroscopic fraction of the population. Reaching a threshold therefore still requires observing order-$N$ many types, even though the observation rates are heterogeneous. One should thus expect the same linear time scale $N$ as in the homogeneous case, together with a deterministic limit profile depending on $s$.

Before proving the main result of this section, we start with an auxiliary lemma about a particular function appearing in the proof of Theorem~\ref{thm:diffuse}.
\begin{lemma}\label{lem:Fs}
Fix $s\in[0,1)$ and define
\[
F_s(\alpha):=
(1-s)\int_0^1 x^{-s}\Bigl(1-e^{-(1-s)\alpha x^{-s}}\Bigr)\,dx,
\qquad
\alpha\ge 0.
\]
Then $F_s$ is continuous, strictly increasing, satisfies $F_s(0)=0$, and
\[
\lim_{\alpha\to\infty}F_s(\alpha)=1.
\]
Hence, for every $\theta\in(0,1)$ there exists a unique $\alpha_{s,\theta}>0$ such that
\[
F_s(\alpha_{s,\theta})=\theta.
\]
\end{lemma}

\begin{proof}
The integrand is continuous in $\alpha$ for every $x\in(0,1]$ and bounded above by $(1-s)x^{-s}$, which is integrable on $(0,1]$. Continuity of $F_s$ therefore follows from dominated convergence. Strict monotonicity follows because for $\alpha_2>\alpha_1$ the integrand at $\alpha_2$ is strictly larger than the integrand at $\alpha_1$ for every $x\in(0,1]$. Clearly $F_s(0)=0$. As $\alpha\to\infty$, the integrand converges pointwise to $(1-s)x^{-s}$, again dominated by the same integrable function, so dominated convergence gives
\[
\lim_{\alpha\to\infty}F_s(\alpha)
=(1-s)\int_0^1 x^{-s}\,dx
=1.
\]
The uniqueness of $\alpha_{s,\theta}$ follows from continuity and strict monotonicity.
\end{proof}

\begin{theorem}\label{thm:diffuse}
Fix $s\in[0,1)$ and $\theta\in(0,1)$. Then
\[
\frac{\tau_{N,\theta}^{(s)}}{N}\xrightarrow[N\to\infty]{\Prob}\frac{\alpha_{s,\theta}}{\lambda},
\]
where $\alpha_{s,\theta}$ is the unique solution of
\[
(1-s)\int_0^1 x^{-s}\Bigl(1-e^{-(1-s)\alpha x^{-s}}\Bigr)\,dx=\theta.
\]
\end{theorem}

\begin{proof}
We apply Theorem~\ref{thm:general-diffuse} with scale $a_N=N/\lambda$. It remains to verify the convergence of the mean profile. At $\alpha=0$,
\[
m_N^{(s)}(0)=0=F_s(0).
\]
Now fix $\alpha>0$. Since
\[
\frac{H_N^{(s)}}{N^{1-s}}\xrightarrow[N \to \infty]{}\frac{1}{1-s},
\]
the quantity
\[
b_N:=\frac{N^{1-s}}{H_N^{(s)}}
\]
satisfies $b_N\xrightarrow[N \to \infty]{} 1-s$.

Writing $x_i=i/N$, and denoting by $m_N^{(s)}(t) = \E[M_N^{(s)}(t)]$, we obtain
\begin{align*}
m_N^{(s)}\!\left(\frac{N\alpha}{\lambda}\right)
&=
\sum_{i=1}^N q_{i,N}^{(s)}\Bigl(1-e^{-\alpha N q_{i,N}^{(s)}}\Bigr)
\\
&=
\frac1N \sum_{i=1}^N
b_N x_i^{-s}\Bigl(1-e^{-\alpha b_N x_i^{-s}}\Bigr).
\end{align*}
Define
\[
g_{N,\alpha}(x):=b_N x^{-s}\Bigl(1-e^{-\alpha b_N x^{-s}}\Bigr),
\qquad x\in(0,1].
\]
Then $g_{N,\alpha}(x)\xrightarrow[N \to \infty]{} g_\alpha(x):=(1-s)x^{-s}\bigl(1-e^{-(1-s)\alpha x^{-s}}\bigr)$ pointwise, and for large $N$,
\[
0\le g_{N,\alpha}(x)\le C x^{-s},
\]
where $Cx^{-s}$ is integrable on $(0,1]$.

Fix $\varepsilon\in(0,1)$. On $[\varepsilon,1]$, the functions $g_{N,\alpha}$ converge uniformly to $g_\alpha$, so
\[
\frac1N \sum_{i=\lceil \varepsilon N\rceil}^N g_{N,\alpha}(i/N)
\xrightarrow[N \to \infty]{}
\int_\varepsilon^1 g_\alpha(x)\,dx.
\]
For the lower tail,
\[
0\le \frac1N \sum_{i=1}^{\lceil \varepsilon N\rceil-1} g_{N,\alpha}(i/N)
\le
\frac{C}{N}\sum_{i=1}^{\lceil \varepsilon N\rceil-1}(i/N)^{-s},
\]
and the right-hand side is bounded by a constant multiple of $\int_0^\varepsilon x^{-s}\,dx$, which tends to $0$ as $\varepsilon\downarrow 0$. The same bound controls $\int_0^\varepsilon g_\alpha(x)\,dx$. Hence
\[
\frac1N \sum_{i=1}^N g_{N,\alpha}(i/N)\xrightarrow[N \to \infty]{} \int_0^1 g_\alpha(x)\,dx = F_s(\alpha).
\]
Hence
\[
m_N^{(s)}\!\left(\frac{N\alpha}{\lambda}\right)\xrightarrow[N \to \infty]{} F_s(\alpha).
\]
Lemma~\ref{lem:Fs} gives the unique root $\alpha_{s,\theta}$ of $F_s(\alpha)=\theta$, so Theorem~\ref{thm:general-diffuse} applies and yields the claim.
\end{proof}

\begin{remark}
For $s=0$ one recovers the uniform case
\[
F_0(\alpha)=1-e^{-\alpha},
\qquad
\alpha_{0,\theta}=-\log(1-\theta),
\]
so we have
\[
\frac{\tau_{N,\theta}^{(0)}}{N}\xrightarrow[N\to \infty]{\Prob} -\frac{\log(1-\theta)}{\lambda}.
\]
\end{remark}
The following theorem shows a local heterogeneity effect near $s=0$.
\begin{theorem}\label{thm:local-monotonicity}
Fix $\theta\in(0,1)$ and let $\alpha_{s,\theta}$ be defined by
\[
F_s(\alpha_{s,\theta})=\theta.
\]
Then, as $s\downarrow 0$,
\[
\alpha_{s,\theta}
=
\alpha_{0,\theta}
+
\frac12\,\alpha_{0,\theta}\bigl(\alpha_{0,\theta}-2\bigr)s^2
+o(s^2),
\qquad
\alpha_{0,\theta}=-\log(1-\theta).
\]
In particular, with
\[
\theta_c:=1-e^{-2},
\]
one has:
\begin{itemize}
\item if $\theta>\theta_c$, then $\alpha_{s,\theta}>\alpha_{0,\theta}$ for all sufficiently small $s>0$;
\item if $\theta<\theta_c$, then $\alpha_{s,\theta}<\alpha_{0,\theta}$ for all sufficiently small $s>0$.
\end{itemize}
\end{theorem}

\begin{proof}
Define
\[
\mathcal{F}(s,\alpha):=
(1-s)\int_0^1 x^{-s}\Bigl(1-e^{-(1-s)\alpha x^{-s}}\Bigr)\,dx,
\]
so that $\mathcal{F}(s,\alpha)=F_s(\alpha)$ for $s\in[0,1)$. With the change of variables $x=e^{-u}$,
\[
\mathcal{F}(s,\alpha)
=
\int_0^\infty G(s,u;\alpha)\,du,
\]
where
\[
G(s,u;\alpha):=(1-s)e^{-(1-s)u}\Bigl(1-e^{-(1-s)\alpha e^{su}}\Bigr).
\]
For $\alpha$ in a compact interval and $|s|\le 1/4$, all partial derivatives of $G(s,u;\alpha)$ up to order two in the variables $(s,\alpha)$ are dominated by a constant multiple of $e^{-u/2}(1+u^2)$. Thus $\mathcal F$ is $C^2$ in a neighborhood of $(0,\alpha_{0,\theta})$, and differentiation under the integral is justified by dominated
convergence. At $s=0$,
\[
\mathcal{F}(0,\alpha)=1-e^{-\alpha},
\qquad
\partial_\alpha \mathcal{F}(0,\alpha)=e^{-\alpha}.
\]
Moreover,
\[
\partial_s G(0,u;\alpha)
=
e^{-u}(u-1)\bigl(1-e^{-\alpha}+\alpha e^{-\alpha}\bigr),
\]
so
\[
\partial_s \mathcal{F}(0,\alpha)=0
\]
because $\int_0^\infty e^{-u}(u-1)\,du=0$. A second differentiation gives
\[
\partial_{ss}G(0,u;\alpha)
=
e^{-u}\Bigl[(u^2-2u)\bigl(1-e^{-\alpha}+\alpha e^{-\alpha}\bigr)
+
\alpha(2-\alpha)e^{-\alpha}(u-1)^2\Bigr].
\]
Using
\[
\int_0^\infty e^{-u}(u^2-2u)\,du=0,
\qquad
\int_0^\infty e^{-u}(u-1)^2\,du=1,
\]
we obtain
\[
\partial_{ss}\mathcal{F}(0,\alpha)=\alpha(2-\alpha)e^{-\alpha}.
\]
Since
\[
\mathcal{F}(0,\alpha_{0, \theta})=\theta
\qquad\text{and}\qquad
\partial_\alpha\mathcal{F}(0,\alpha_{0, \theta})=e^{-\alpha_{0, \theta}}>0,
\]
the implicit function theorem yields a $C^2$ map $s\mapsto \alpha_{s,\theta}$ near $0$ such that
\[
\mathcal{F}(s,\alpha_{s,\theta})=\theta.
\]
Differentiating once gives
\[
\frac{d}{ds}\alpha_{s,\theta}=-\frac{\partial_s\mathcal{F}(s,\alpha_{s,\theta})}{\partial_\alpha\mathcal{F}(s,\alpha_{s,\theta})},
\]
hence $\frac{d}{ds} \alpha_{s,\theta}|_{s = 0}=0$. Differentiating a second time and evaluating at $s=0$ yields
\[
\frac{d^2}{ds^2}\alpha_{s,\theta}\Big|_{s = 0}
=
-\frac{\partial_{ss}\mathcal{F}(0,\alpha_{0, \theta})}{\partial_\alpha\mathcal{F}(0,\alpha_{0, \theta})}
=
\alpha_{0, \theta}(\alpha_{0, \theta}-2).
\]
Taylor's formula therefore gives
\[
\alpha_{s,\theta}
=
\alpha_{0, \theta}+\frac12\,\alpha_{0, \theta}(\alpha_{0, \theta}-2)s^2+o(s^2).
\]
Since $\alpha_{0, \theta}>2$ is equivalent to $\theta>1-e^{-2}$, the sign conclusions follow immediately.
\end{proof}

\section{The Critical Zipf Regime: $s=1$} \label{sec:The_Critical_Aligned_Zipf_Regime}
For the harmonic aligned Zipf family
\[
q_{i,N}^{(1)}=\frac{1}{iH_N},
\qquad
H_N = H_N^{(1)} =\sum_{j=1}^N \frac1j,
\]
with
\[
p_i=w_i=q_{i,N}^{(1)},
\qquad i\in[N],
\]
the natural scale is no longer linear in $N$. Instead, the leading order depends on the threshold through a power law. More precisely, the appropriate scaling of the quorum time is $N^\theta H_N$. A simple heuristic explains the scale $N^\theta H_N$. At time $t=H_NA/\lambda$, type $i$ has been seen with probability $1-e^{-A/i}$. Thus, types with $i\lesssim A$ have essentially already appeared, while types with $i\gg A$ contribute only weakly. The discovered weight is therefore roughly the harmonic mass of the first $A$ types,
\[
\frac{1}{H_N}\sum_{i\le A}\frac1i
\approx \frac{\log A}{\log N}.
\]
To make this mass equal to the threshold $\theta$, one should choose $\log A\approx \theta\log N$, that is $A\approx N^\theta$. Since $t=H_NA/\lambda$, this predicts the critical scale $N^\theta H_N$. Theorem~\ref{thm:critical} below makes this heuristic precise and also identifies the leading constant. We first show the mean asymptotics in the critical case $s = 1$.

\begin{lemma}\label{lem:critical-mean}
Fix $\beta\in(0,1)$ and $c>0$. Then
\[
m_N^{(1)}\!\left(\frac{cH_N N^\beta}{\lambda}\right)
=
\beta+\frac{\log c+(2-\beta)\gamma}{H_N}+o\!\left(\frac1{H_N}\right),
\]
where $\gamma$ denotes Euler's constant.
\end{lemma}

\begin{proof}
Set
\begin{equation}\label{eq:def_of_A_N_and_t_N}
    A_N:=cN^\beta, \qquad t_N:=\frac{H_NA_N}{\lambda}.
\end{equation}
Then, by \eqref{eq:mean_profile}
\[
m_N^{(1)}(t_N)
=
\frac1{H_N}\sum_{i=1}^N \frac{1-e^{-A_N/i}}{i}.
\]
Write
\[
S_N:=\sum_{i=1}^N \frac{1-e^{-A_N/i}}{i}.
\]
Using
\[
1-e^{-A_N/i}=\int_0^{A_N/i} e^{-u}\,du
\]
and Tonelli's theorem, we obtain
\[
S_N
=
\int_0^\infty e^{-u}\sum_{i=1}^N \frac1i \1_{\{u\le A_N/i\}}\,du.
\]
Let $H_0:=0$. Since $A_N/N\to 0$,
\[
S_N
=
\int_0^{A_N/N} e^{-u}H_N\,du
+
\int_{A_N/N}^{A_N} e^{-u}H_{\lfloor A_N/u\rfloor}\,du.
\]
The first term is bounded by
\[
H_N\frac{A_N}{N}=cH_NN^{\beta-1}=o(1),
\]
since $\beta < 1$. For $u\in[A_N/N,A_N]$, one has $\lfloor A_N/u\rfloor\ge 1$, and the harmonic expansion
\[
H_m=\log m+\gamma+O(m^{-1})
\]
combined with
\[
\log \lfloor A_N/u\rfloor=\log(A_N/u)+O(u/A_N)
\]
gives
\[
H_{\lfloor A_N/u\rfloor}
=
\log(A_N/u)+\gamma+O(u/A_N),
\]
uniformly on that interval. Hence
\[
S_N
=
\int_{A_N/N}^{A_N} e^{-u}\bigl(\log(A_N/u)+\gamma\bigr)\,du
+O\!\left(\frac1{A_N}\int_0^{A_N} ue^{-u}\,du\right)+o(1).
\]
Since $A_N\to\infty$, the second term is also $o(1)$. Furthermore,
\begin{equation}\label{eq:int_0^A_N/N}
    \int_0^{A_N/N} e^{-u}\bigl|\log(A_N/u)+\gamma\bigr|\,du
\le \int_0^{A_N/N}\big( |\log A_N| + |\log u| + \gamma \big) \,du.
\end{equation}
It holds that
\begin{equation*}
    \int_0^{\varepsilon} |\log u|\,du = \varepsilon (1 - \log \varepsilon), \quad \text{for}\ \varepsilon \in (0, 1).
\end{equation*}
Setting $\varepsilon = A_N/N = cN^{\beta - 1} \xrightarrow[N\to \infty]{} 0$, we have (for $N$ large enough)
\begin{equation}\label{eq:log_u_part}
    \int_0^{A_N/N} |\log u|\, du = \frac{A_N}{N}\left( 1 + \log \frac{N}{A_N} \right) = cN^{\beta - 1}(1 + \log (c^{-1} N^{1-\beta})) = o(1),
\end{equation}
because $\beta<1$. Similarly
\begin{equation}\label{eq:log_AN_and_gamma_part}
    \frac{A_N}{N} (|\log(A_N)| + \gamma) = o(1).
\end{equation}
Plugging \eqref{eq:log_u_part} and \eqref{eq:log_AN_and_gamma_part} back to \eqref{eq:int_0^A_N/N}, and noticing that the tail integral over $[A_N, \infty)$ is exponentially small, we get
\[
S_N
=
\int_0^\infty e^{-u}\bigl(\log(A_N/u)+\gamma\bigr)\,du+o(1).
\]
Using
\[
\int_0^\infty e^{-u}\,du=1
\qquad\text{and}\qquad
\int_0^\infty e^{-u}\log u\,du=-\gamma,
\]
we conclude that
\[
S_N=\log A_N+2\gamma+o(1).
\]
Since
\[
\log A_N=\beta\log N+\log c
\]
and
\[
H_N=\log N+\gamma+o(1),
\]
it follows that
\[
m_N^{(1)}(t_N)
=
\frac{\beta\log N+\log c+2\gamma+o(1)}{H_N}
=
\beta+\frac{\log c+(2-\beta)\gamma}{H_N}+o\!\left(\frac1{H_N}\right).
\]
\end{proof}

After showing the mean asymptotics, we also provide a variance bound in the critical case.
\begin{lemma}[Critical Variance Bound]\label{lem:critical-var}
Fix $\beta\in(0,1)$ and $c>0$. Then
\[
\Var\!\left(M_N^{(1)}\!\left(\frac{cH_NN^\beta}{\lambda}\right)\right)
=
O\!\left(\frac{1}{H_N^2N^\beta}\right).
\]
\end{lemma}
\begin{proof}
Let $A_N = cN^\beta$ and $t_N = (H_N A_N)/\lambda$, as in \eqref{eq:def_of_A_N_and_t_N}.
Then, by \eqref{eq:Var_MNt}
\[
\Var\bigl(M_N^{(1)}(t_N)\bigr)
=
\frac1{H_N^2}\sum_{i=1}^N \frac{e^{-A_N/i}\bigl(1-e^{-A_N/i}\bigr)}{i^2}.
\]
We split the sum at $\lfloor A_N\rfloor$. For $i\le A_N$ we use the fact that
\[
e^{-A_N/i}\le C\left(\frac{i}{A_N}\right)^2
\]
with a universal constant $C$. Hence
\[
\sum_{i\le A_N}\frac{e^{-A_N/i}\bigl(1-e^{-A_N/i}\bigr)}{i^2}
\le
\frac{C}{A_N^2}\sum_{i\le A_N}1
=
O\!\left(\frac1{A_N}\right).
\]
For $i>A_N$ we use $1-e^{-A_N/i}\le A_N/i$, so
\[
\sum_{i>A_N}\frac{e^{-A_N/i}\bigl(1-e^{-A_N/i}\bigr)}{i^2}
\le
A_N\sum_{i>A_N}\frac1{i^3}
=
O\!\left(\frac1{A_N}\right).
\]
Combining the two bounds gives
\[
\Var\bigl(M_N^{(1)}(t_N)\bigr)=O\!\left(\frac1{H_N^2A_N}\right)
=
O\!\left(\frac1{H_N^2N^\beta}\right).
\]
\end{proof}

We are now ready to show the law of large numbers in the case $s = 1$.
\begin{theorem}\label{thm:critical}
Fix $\theta\in(0,1)$ and define
\[
c_\theta:=e^{(\theta-2)\gamma}.
\]
Then
\[
\frac{\tau_{N,\theta}^{(1)}}{H_NN^\theta}
\xrightarrow[N\to\infty]{\Prob}
\frac{c_\theta}{\lambda}.
\]
\end{theorem}

\begin{proof}
Fix $\varepsilon\in(0,1)$ and set
\[
c_-:=(1-\varepsilon)c_\theta,
\qquad
c_+:=(1+\varepsilon)c_\theta,
\]
\[
t_N^-:=\frac{c_-H_NN^\theta}{\lambda},
\qquad
t_N^+:=\frac{c_+H_NN^\theta}{\lambda}.
\]
Lemma~\ref{lem:critical-mean} with $\beta=\theta$ yields
\[
m_N^{(1)}(t_N^\pm)
=
\theta+\frac{\log c_\pm+(2-\theta)\gamma}{H_N}
+o\!\left(\frac1{H_N}\right)
=
\theta+\frac{\log(c_\pm/c_\theta)}{H_N}
+o\!\left(\frac1{H_N}\right).
\]
Since $\log(c_-/c_\theta)<0<\log(c_+/c_\theta)$, there exists $\kappa>0$ such that, for all sufficiently large $N$,
\[
m_N^{(1)}(t_N^-)\le \theta-\frac{\kappa}{H_N},
\qquad
m_N^{(1)}(t_N^+)\ge \theta+\frac{\kappa}{H_N}.
\]
Therefore
\[
\Prob\!\left(\tau_{N,\theta}^{(1)}\le t_N^-\right)
\le
\Prob\!\left(\left|M_N^{(1)}(t_N^-)-m_N^{(1)}(t_N^-)\right|\ge \frac{\kappa}{H_N}\right)
\]
and similarly
\[
\Prob\!\left(\tau_{N,\theta}^{(1)} > t_N^+\right)
\le
\Prob\!\left(\left|M_N^{(1)}(t_N^+)-m_N^{(1)}(t_N^+)\right|\ge \frac{\kappa}{H_N}\right).
\]
Chebyshev's inequality and Lemma~\ref{lem:critical-var} imply that both probabilities are
\[
O(N^{-\theta})\xrightarrow[N\to \infty]{} 0.
\]
Hence
\[
\Prob\!\left(t_N^-\le\tau_{N,\theta}^{(1)}\le t_N^+\right)\longrightarrow 1.
\]
Since
\[
t_N^\pm=\frac{c_\theta(1\pm\varepsilon)H_NN^\theta}{\lambda},
\]
we obtain
\[
\Prob\!\left(
\frac{(1-\varepsilon)c_\theta}{\lambda}
\le
\frac{\tau_{N,\theta}^{(1)}}{H_NN^\theta}
\le
\frac{(1+\varepsilon)c_\theta}{\lambda}
\right)\longrightarrow 1.
\]
As $\varepsilon>0$ was arbitrary, the claim follows.
\end{proof}
In the following corollary we identify the expression for which the threshold parameter $\theta$ is reconstructed in the limit.
\begin{corollary}\label{cor:critical-log}
Fix $\theta\in(0,1)$. Then
\begin{equation}\label{eq:theta_cor_1st}
    \frac{\log\bigl(\lambda \tau_{N,\theta}^{(1)}/H_N\bigr)}{\log N}
\xrightarrow[N\to\infty]{\Prob}
\theta.    
\end{equation}
In particular,
\begin{equation}\label{eq:theta_cor_2nd}
    \frac{\log \tau_{N,\theta}^{(1)}}{\log N}
\xrightarrow[N\to\infty]{\Prob}
\theta.
\end{equation}
\end{corollary}

\begin{proof}
    Set
    \begin{equation*}
        X_N := \frac{\lambda \tau_{N, \theta}^{(1)}}{H_N N^{\theta}}.
    \end{equation*}
    Theorem \ref{thm:critical} states exactly that
    \begin{equation*}
        X_N \xrightarrow[N\to \infty]{\Prob} c_{\theta},
    \end{equation*}
    where $c_{\theta} = e^{(\theta - 2)\gamma} > 0$. Since $\tau_{N, \theta}^{(1)} > 0$ almost surely, we have $X_N > 0$ almost surely, so the logarithm is well defined. Because $\log$ is continuous on $(0, \infty)$, the continuous mapping theorem gives
    \begin{equation*}
        \log X_N \xrightarrow[N \to \infty]{\Prob} \log c_{\theta} = (\theta - 2)\gamma.
    \end{equation*}
    Hence,
    \begin{equation}
        \frac{\log X_N}{\log N}\xrightarrow[N\to \infty]{\Prob} 0.
    \end{equation}
    Now we have
    \begin{equation*}
        \frac{\log \bigl(\lambda \tau_{N,\theta}^{(1)}/H_N\bigr)}{\log N} = \frac{\log(N^\theta X_N)}{\log N} = \theta + \frac{\log X_N}{\log N} \xrightarrow[N\to \infty]{\Prob} \theta,
    \end{equation*}
    which is precisely \eqref{eq:theta_cor_1st}. For \eqref{eq:theta_cor_2nd}, we first notice that
    \begin{equation*}
        \log \tau_{N, \theta}^{(1)} = \theta \log N + \log H_N - \log \lambda + \log X_N.
    \end{equation*}
    Therefore,
    \begin{equation*}
        \frac{\log \tau_{N, \theta}^{(1)}}{\log N} = \theta + \frac{\log H_N}{\log N} - \frac{\log \lambda}{\log N} + \frac{\log X_N}{\log N} \xrightarrow[N\to \infty]{\Prob} \theta.
    \end{equation*}
\end{proof}

\section{The Atomic Zipf Regime: $s>1$}\label{sec:The_Atomic_Aligned_Zipf_Regime}
Fix $s>1$ and set
\[
q_{i,N}^{(s)}=\frac{i^{-s}}{H_N^{(s)}}\1_{\{i\le N\}},
\qquad
q_i^{(s)}:=\frac{i^{-s}}{\zeta(s)},
\qquad i\ge 1,
\]
where $\zeta(s) = \sum_{n = 1}^{\infty}n^{-s}$ is the Riemann zeta function. Clearly $q_{i,N}^{(s)} \xrightarrow[N\to\infty]{} q_i^{(s)}$ for every $i \ge 1$, so the weights are atomic in the sense of Definition~\ref{def:diffuse_vs_atomic}. For simplicity, in this section, we introduce a sequence $(E_i)_{i \in \N}$ of independent and identically distributed $\mathrm{Exp}(1)$ random variables. Using this sequence we define
\[
M_\infty^{(s)}(t):=\sum_{i=1}^{\infty} q_i^{(s)} \1_{\{E_i\le \lambda q_i^{(s)} t\}}.
\]
Notice that the definition of $M_N^{(s)}(t)$ from \eqref{eq:def_of_discovered-mass_process} can be rewritten in terms of sequences $(E_i)_{i \in \N}$ and $(q_{i,N}^{(s)})_{i \in \N}$ as
\begin{equation*}
    M_N^{(s)}(t) = \sum_{i=1}^\infty q_{i,N}^{(s)} \1_{\{E_i\le \lambda q_{i,N}^{(s)} t\}},
\end{equation*}
since $E_i / (\lambda q_{i,N}^{(s)}) \sim \mathrm{Exp}(\lambda q_{i,N}^{(s)})$, and $q_{i,N}^{(s)} = 0$ for $i > N$. Both sums are almost surely absolutely convergent for every fixed $t$ because each summand is bounded by the corresponding weight and the weights are summable. 

Before proving the main result of this section (Theorem \ref{thm:atomic-hitting}), we again provide some heuristics. The regime $s>1$ is qualitatively different from the previous regimes because a few leading types now carry non-vanishing quorum weight even as $N\to\infty$. Quorum formation is therefore strongly influenced by the random arrival times of these persistent heavy types. Unlike the diffuse regime, averaging over many tiny contributions no longer occurs, so one should expect genuine randomness to survive in the limit and the quorum time to converge to a non-degenerate random object. A further subtlety appears in the atomic regime. Since the limiting process is purely atomic, its sample paths are step functions with positive jump sizes. As a result, the threshold may be reached exactly at a jump time with positive probability. The natural limiting object is therefore not the hitting time with a strict inequality
\[
\inf\{t\ge0:M_\infty^{(s)}(t)>\theta\},
\]
but rather
\[
\inf\{t\ge0:M_\infty^{(s)}(t)\ge\theta\}.
\]
This is precisely the notion that emerges from the finite-$N$ strict threshold rule in the atomic limit. We first give fixed time bounds in the atomic setting.

\begin{proposition}\label{prop:atomic-fixed-time}
Fix $s>1$ and $t > 0$. Then
\[
\E\left[\left|M_N^{(s)}(t)-M_\infty^{(s)}(t)\right|\right]
\le
(2+\lambda t)\sum_{i>N} q_i^{(s)}.
\]
\end{proposition}

\begin{proof}
For $i\le N$, write
\[
A_i(t):=q_{i,N}^{(s)}\1_{\{E_i\le \lambda q_{i,N}^{(s)} t\}},
\qquad
B_i(t):=q_i^{(s)}\1_{\{E_i\le \lambda q_i^{(s)} t\}}.
\]
Since $H_N^{(s)}<\zeta(s)$, one has $q_{i,N}^{(s)}\ge q_i^{(s)}$ for every $i\le N$. Therefore
\[
A_i(t)\ge B_i(t)
\qquad\text{almost surely.}
\]
More precisely,
\[
|A_i(t)-B_i(t)|
=(q_{i,N}^{(s)}-q_i^{(s)})\1_{\{E_i\le \lambda q_i^{(s)} t\}}
+q_{i,N}^{(s)}\1_{\{\lambda q_i^{(s)} t<E_i\le \lambda q_{i,N}^{(s)} t\}}.
\]
Taking expectations and using
\[
\Prob\!\left(\lambda q_i^{(s)} t<E_i\le \lambda q_{i,N}^{(s)} t\right)
\le
\lambda t\bigl(q_{i,N}^{(s)}-q_i^{(s)}\bigr),
\]
we obtain
\[
\E\left[|A_i(t)-B_i(t)|\right]
\le
\bigl(q_{i,N}^{(s)}-q_i^{(s)}\bigr)
+\lambda t\,q_{i,N}^{(s)}\bigl(q_{i,N}^{(s)}-q_i^{(s)}\bigr)
\le
(1+\lambda t)\bigl(q_{i,N}^{(s)}-q_i^{(s)}\bigr),
\]
because $q_{i,N}^{(s)}\le 1$. Summing over $i\le N$ and adding the tail $i>N$, we get
\begin{align*}
\E\left[\left|M_N^{(s)}(t)-M_\infty^{(s)}(t)\right|\right]
&\le
\sum_{i=1}^N (1+\lambda t)\bigl(q_{i,N}^{(s)}-q_i^{(s)}\bigr)
+\sum_{i>N} q_i^{(s)}.
\end{align*}
Moreover,
\[
\sum_{i=1}^N \bigl(q_{i,N}^{(s)}-q_i^{(s)}\bigr)
=
1-\sum_{i=1}^N q_i^{(s)}
=
\sum_{i>N} q_i^{(s)}.
\]
Substituting this identity yields the claimed bound.
\end{proof}

\begin{remark}
The same argument yields finite-dimensional convergence in the atomic regime: for any fixed times $t_1,\dots,t_m\ge 0$,
\[
\sum_{j=1}^m \E\left[\left|M_N^{(s)}(t_j)-M_\infty^{(s)}(t_j)\right|\right]\longrightarrow 0.
\]
A functional limit theorem for the full path is not used here; the hitting-time convergence below follows from a continuity-point argument.
\end{remark}

\begin{corollary}\label{cor:atomic}
Fix $s>1$ and $t>0$. Then
\[
M_N^{(s)}(t)\xrightarrow[N\to\infty]{L^1} M_\infty^{(s)}(t),
\]
and the limiting random variable is non-degenerate,
\[
\Var\bigl(M_\infty^{(s)}(t)\bigr)\ge
\bigl(q_1^{(s)}\bigr)^2 e^{-\lambda q_1^{(s)} t}\bigl(1-e^{-\lambda q_1^{(s)} t}\bigr)>0.
\]
In particular, the atomic regime does not collapse to a deterministic discovered-mass curve.
\end{corollary}

\begin{proof}
The $L^1$ convergence follows directly from Proposition~\ref{prop:atomic-fixed-time}. For the variance bound, note that
\[
M_\infty^{(s)}(t)=\sum_{i=1}^\infty q_i^{(s)}\1_{\{E_i\le \lambda q_i^{(s)} t\}}
\]
is a sum of independent Bernoulli-weighted terms, so
\[
\Var\bigl(M_\infty^{(s)}(t)\bigr)
=
\sum_{i=1}^\infty \bigl(q_i^{(s)}\bigr)^2 e^{-\lambda q_i^{(s)} t}\bigl(1-e^{-\lambda q_i^{(s)} t}\bigr).
\]
The first summand already gives the stated strictly positive lower bound.
\end{proof}

We now find the limit of the quorum time in the atomic setting.
\begin{theorem}\label{thm:atomic-hitting}
Fix $s>1$ and $\theta\in(0,1)$. Define the limiting hitting time
\[
\sigma_{\infty,\theta}^{(s)}:=\inf\{t\ge 0:M_\infty^{(s)}(t)\ge \theta\}.
\]
Then, under the common coupling above,
\[
\tau_{N,\theta}^{(s)}\xrightarrow[N\to\infty]{a.s.}\sigma_{\infty,\theta}^{(s)}.
\]
\end{theorem}
\begin{proof}[Proof of Theorem \ref{thm:atomic-hitting}]
We first record two almost-sure properties of the limiting process. For every $T>0$,
\[
\sum_{i=1}^\infty \Prob\!\left(E_i\le \lambda q_i^{(s)} T\right)
=
\sum_{i=1}^\infty \bigl(1-e^{-\lambda q_i^{(s)} T}\bigr)
\le
\lambda T\sum_{i=1}^\infty q_i^{(s)}
=
\lambda T<\infty.
\]
Hence Borel--Cantelli implies that, for each fixed $T$, almost surely only finitely many indices $i$ satisfy $E_i\le \lambda q_i^{(s)}T$. To make this simultaneous over compact intervals, set
\[
\Omega_0:=
\bigcap_{m=1}^\infty
\left\{
\#\{i:E_i\le \lambda q_i^{(s)}m\}<\infty
\right\}.
\]
Then $\Prob(\Omega_0)=1$, and on $\Omega_0$ the process $M_\infty^{(s)}$ has only finitely many jumps on every compact time interval. Since each indicator in the definition of $M_\infty^{(s)}(t)$ converges monotonically to $1$ as $t\to\infty$, one also has
\[
M_\infty^{(s)}(t)\uparrow \sum_{i=1}^\infty q_i^{(s)}=1
\qquad\text{almost surely,}
\]
and therefore $\sigma_{\infty,\theta}^{(s)}<\infty$ almost surely. Fix a sample point
$\omega\in\Omega_0$ for which this convergence to $1$ holds and for which $E_i(\omega)>0$
for all $i$, and write
\[
f_N(t):=M_N^{(s)}(t,\omega),
\qquad
f(t):=M_\infty^{(s)}(t,\omega).
\]
Let $t\ge 0$ be a continuity point of $f$. We claim that
\begin{equation}\label{eq:f_N_to_f}
    f_N(t)\xrightarrow[N\to \infty]{} f(t).
\end{equation}
Indeed, fix $\varepsilon>0$ and choose $K$ so large that
\[
2\sum_{i>K} i^{-s}<\varepsilon.
\]
Since $q_{i,N}^{(s)}\le i^{-s}$ and $q_i^{(s)}\le i^{-s}$, this implies
\begin{equation}\label{eq:tail_bound_q}
    \sum_{i>K} q_{i,N}^{(s)}+\sum_{i>K} q_i^{(s)}<\varepsilon.
\end{equation}
This holds for all $N \in \N$. Notice that the first sum is zero if $N \le K$. For the first $K$ terms (i.e.\ for $i \le K$), continuity of $f$ at $t$ means that
\[
t\neq \frac{E_i(\omega)}{\lambda q_i^{(s)}},
\qquad\text{for } i=1,\dots,K,
\]
so for each such $i$ one has
\begin{equation}\label{eq:q_iN_cvg_to_q_i}
q_{i,N}^{(s)}\1_{\{E_i(\omega)\le \lambda q_{i,N}^{(s)} t\}}
\xrightarrow[N\to \infty]{}
q_i^{(s)}\1_{\{E_i(\omega)\le \lambda q_i^{(s)} t\}}.
\end{equation}
We now fix $K$ to be large enough so that \eqref{eq:tail_bound_q} holds. For such $K$ we have
\[
\left|
f_N(t)-f(t)
\right|
\le
\sum_{i=1}^K
\left|
q_{i,N}^{(s)}\1_{\{E_i(\omega)\le \lambda q_{i,N}^{(s)} t\}}
-
q_i^{(s)}\1_{\{E_i(\omega)\le \lambda q_i^{(s)} t\}}
\right|
\,+\,\varepsilon.
\]
Using \eqref{eq:q_iN_cvg_to_q_i}, we get that the finite sum tends to $0$. Since $\varepsilon > 0$ was arbitrary, this proves the claim \eqref{eq:f_N_to_f}.

Now, for the simplicity of notation, let
\[
\sigma=\sigma_{\infty,\theta}^{(s)}(\omega).
\]
Since $f(0)=0<\theta$ and $f$ has only finitely many jumps on compact intervals, one has $\sigma>0$. Fix $\varepsilon>0$. Since $f$ is nondecreasing and has at most finitely many jumps on $[\max\{\sigma-\varepsilon,0\},\sigma+\varepsilon]$, one can choose continuity points
\[
a\in(\max\{\sigma-\varepsilon,0\},\sigma),
\qquad
b\in(\sigma,\sigma+\varepsilon)
\]
such that
\[
f(a)<\theta.
\]
By the continuity-point convergence proved above (see \eqref{eq:f_N_to_f}),
\[
f_N(a)\xrightarrow[N \to \infty]{} f(a).
\]
Hence, for all sufficiently large $N$,
\[
f_N(a)<\theta.
\]
At the upper time $b$, there are two cases.

If $f(b)>\theta$, then continuity-point convergence gives $f_N(b)\to f(b)$, hence $f_N(b)>\theta$ for all sufficiently large $N$.

If $f(b)=\theta$, let
\[
A_b(\omega):=\{i\in\N:E_i(\omega)\le \lambda q_i^{(s)}b\}.
\]
Since $b<\infty$ and $f$ has only finitely many jumps on compact intervals, the set $A_b(\omega)$ is finite. For every $N\ge \max A_b(\omega)$, the identity
\[
q_{i,N}^{(s)}=\frac{\zeta(s)}{H_N^{(s)}}\,q_i^{(s)}=:c_Nq_i^{(s)},
\qquad c_N>1,
\]
implies
\[
f_N(b)
\ge
\sum_{i\in A_b(\omega)} q_{i,N}^{(s)}
=
c_N\sum_{i\in A_b(\omega)} q_i^{(s)}
=
c_N f(b)
=
c_N\theta
>
\theta.
\]
Hence, in either case $f_N(b)>\theta$ for all sufficiently large $N$. By monotonicity of $f_N$, this implies
\[
a<\tau_{N,\theta}^{(s)}(\omega)\le b.
\]
Therefore
\[
|\tau_{N,\theta}^{(s)}(\omega)-\sigma|<\varepsilon
\]
for all sufficiently large $N$. Since $\varepsilon>0$ was arbitrary,
\begin{equation*}
    \tau_{N,\theta}^{(s)}(\omega)\xrightarrow[N\to\infty]{} \sigma.    
\end{equation*}
The argument holds on an event of probability one, so the convergence is almost sure.
\end{proof}

\begin{corollary}\label{cor:atomic-expectation}
Fix $s>1$ and $\theta\in(0,1)$. Then
\[
\tau_{N,\theta}^{(s)}\xrightarrow[N\to \infty]{L^1} \sigma_{\infty,\theta}^{(s)}.
\]
\end{corollary}

\begin{proof}
Let
\[
K_\theta:=\min\left\{k\in\N:\sum_{i=1}^k q_i^{(s)}\ge \theta\right\},
\qquad
Y_\theta:=\max_{1\le i\le K_\theta}\frac{E_i}{\lambda q_i^{(s)}}.
\]
Since $q_{i,N}^{(s)}\ge q_i^{(s)}$ for $i\le N$, one has for every $N\ge K_\theta$,
\[
\tau_{N,\theta}^{(s)}
\le
\max_{1\le i\le K_\theta}\frac{E_i}{\lambda q_{i,N}^{(s)}}
\le
Y_\theta.
\]
The same bound implies
\[
\sigma_{\infty,\theta}^{(s)}\le Y_\theta.
\]
Because $Y_\theta$ is integrable and Theorem~\ref{thm:atomic-hitting} gives almost-sure convergence, dominated convergence yields the claim.
\end{proof}

\section{Exact Non-Monotonicity in the Smallest Nontrivial System}\label{sec:Exact_Non-Monotonicity}
The non-monotonicity phenomenon already appears in the aligned Zipf model with $N=2$.

\begin{theorem}\label{thm:n2}
Consider the Zipf family with $N=2$:
\[
w_1(s)=\frac{1}{1+2^{-s}},
\qquad
w_2(s)=\frac{2^{-s}}{1+2^{-s}},
\qquad s\ge 0.
\]
Fix $\theta\in(1/2,1)$ and set
\[
s_\theta:=\log_2\!\left(\frac{\theta}{1-\theta}\right).
\]
Then the expected quorum time has the explicit form
\[
\E[\tau_{2,\theta}^{(s)}]
=
\begin{cases}
\dfrac{1}{\lambda}\!\left(\dfrac{1}{w_1(s)}+\dfrac{1}{w_2(s)}-1\right), & 0\le s\le s_\theta,\\[1.0em]
\dfrac{1}{\lambda\, w_1(s)}, & s>s_\theta.
\end{cases}
\]
Moreover:
\begin{itemize}
\item the map $s\mapsto \E[\tau_{2,\theta}^{(s)}]$ is strictly increasing on $[0,s_\theta]$;
\item the map $s\mapsto \E[\tau_{2,\theta}^{(s)}]$ is strictly decreasing on $(s_\theta,\infty)$.
\end{itemize}
Hence global monotonicity in the aligned heterogeneity parameter is false.
\end{theorem}

\begin{proof}
Since $\theta>1/2$ and $w_1(s)\ge w_2(s)$, the quorum rule is as follows.

If $w_1(s)\le \theta$, then one type is not enough and both types must be observed. Since
\[
w_1(s)=\frac{1}{1+2^{-s}},
\]
the inequality $w_1(s)\le \theta$ is equivalent to $s\le s_\theta$. In that regime,
\[
\tau_{2,\theta}^{(s)}=\max\{T_1,T_2\},
\]
where $T_i\sim \mathrm{Exp}(\lambda w_i(s))$ independently. The expectation of the maximum of two independent exponentials is
\[
\E[\max\{T_1,T_2\}]
=
\E[T_1]+\E[T_2]-\E[\min\{T_1,T_2\}]
=
\frac{1}{\lambda w_1(s)}+\frac{1}{\lambda w_2(s)}-\frac{1}{\lambda},
\]
which gives the first formula.

If $w_1(s)>\theta$, equivalently $s>s_\theta$, then the first type alone already exceeds the threshold, so
\[
\tau_{2,\theta}^{(s)}=T_1,
\]
and
\[
\E[\tau_{2,\theta}^{(s)}]=\frac{1}{\lambda w_1(s)}.
\]

It remains to prove the monotonicity claims. The function $w_1(s)$ is clearly strictly increasing and $w_2(s)=1-w_1(s)$. The function
\[
f(w):=\frac{1}{w}+\frac{1}{1-w}-1
\]
satisfies
\[
f'(w)=-\frac{1}{w^2}+\frac{1}{(1-w)^2}>0
\qquad \text{for } w\in(1/2,1).
\]
Thus $f(w_1(s))/\lambda$ is strictly increasing in $s$ on $[0,s_\theta]$. On the second branch,
\[
s \mapsto \frac{1}{\lambda w_1(s)}
\]
is strictly decreasing because $w_1(s)$ is strictly increasing. Therefore the expectation increases and then decreases, so it cannot be globally monotone.
\end{proof}

\section{Discussion}\label{sec:Discussion}
Weighted-threshold collection depends on two separate objects: the activity
profile \(p\), which determines how quickly types are discovered, and the weight
profile \(w\), which determines how much each discovered type contributes to the
threshold.

When \(p_i=1/N\) and \(\max_i w_i\to0\), Theorem~\ref{thm:homogeneous} gives the universal limit. Within this class, the first-order threshold time is independent of the shape of \(w\). In the aligned Zipf family $p_i=w_i\propto i^{-s}$, the behavior depends on the exponent. In the diffuse regime $0\le s<1$, the discovered-mass process concentrates around its mean profile and the quorum time becomes asymptotically deterministic on a linear time scale. The local expansion of Theorem~\ref{thm:local-monotonicity} shows that the effect of a small Zipf perturbation of the uniform case changes sign at \(\theta_c=1-e^{-2}\). 
At the critical point $s=1$, the sharp scale is $H_NN^\theta$, with the explicit constant from Theorem~\ref{thm:critical}. In the atomic regime $s>1$, macroscopic atoms survive, the discovered mass does not self-average, and the quorum time converges for every threshold to the limiting hitting time \(\sigma_{\infty,\theta}^{(s)}\) almost surely (see Theorem \ref{thm:atomic-hitting}) and in $L^1$ (see Corollary~\ref{cor:atomic-expectation}).

Several mathematical directions remain open:
\begin{itemize}
\item extend the aligned Zipf analysis to broader profile pairs $(p^{(N)},w^{(N)})$, especially regularly varying families and perturbations of the homogeneous-clock benchmark;
\item determine the next-order behavior and possible fluctuation theory at the critical scale $H_NN^\theta$;
\item develop limiting fluctuation or distributional results for the closed atomic hitting time $\sigma_{\infty,\theta}^{(s)}$ and its dependence on $s$ and $\theta$;
\item isolate structural conditions on $(p,w)$ or on the threshold level under which monotonicity may still hold, despite the explicit counterexample of Theorem~\ref{thm:n2}.
\end{itemize}

From an applied point of view, the results show that weighted quorum formation
depends not only on the distribution of quorum weights, but also on how these
weights are correlated with activity. If activity is homogeneous and no
individual weight is asymptotically visible, weight heterogeneity has no
first-order effect. In contrast, when activity and weight are aligned, as in the
Zipf model, heterogeneity changes both the scale and the nature of the quorum
time. In particular, sufficiently concentrated weights lead to a regime in
which the discovery times of a few large-weight types remain visible in the
limit. Thus, in weighted quorum systems, concentration of weight can create an
intrinsic randomness in confirmation times that is not removed by increasing the
number of participants.

\section*{Declaration of generative AI and AI-assisted technologies}

\noindent
During the preparation of this work, the authors used generative AI systems as interactive tools for discussing ideas and improving exposition. All mathematical derivations, proofs, and final editorial decisions were carried out and verified by the authors. The authors take full responsibility for the correctness and content of the manuscript.

\section*{Acknowledgments} 

\noindent
Financial support through the \emph{Croatian Science Foundation} under project IP-2022-10-2277 (for S.\ \v{S}ebek) is gratefully acknowledged. This research was also funded by the European union--NextGenerationEU through the National Recovery and Resilience Plan 2021-2026 Institutional grant of University of Zagreb Faculty of Electrical Engineering and Computing (VALOR). This work was carried out within a project DIGIT.2.1.02.016 funded by the Digital, Innovation, and Green Technology Project – DIGIT Project (IBRD Loan No. 9558‑HR).

\bibliographystyle{plain}
\bibliography{literature.bib}

\end{document}